\documentclass{amsart}
\usepackage[top=1in, bottom=1in, left=1.25in, right=1.25in]{geometry}
\usepackage{mathrsfs}
\usepackage{amssymb}
\usepackage{amsfonts}
\usepackage[linkcolor=red,citecolor=blue]{hyperref}
\usepackage{MnSymbol}
\usepackage{latexsym}
\usepackage[all,cmtip]{xy}
\usepackage{amsmath,amsthm}
\usepackage{color,ifsym}
\usepackage{graphicx}
\usepackage{fancybox}
\usepackage{longtable}
\usepackage{lscape}
\usepackage{float}
\usepackage{mathdots}
\usepackage{enumerate}
\usepackage{multirow}

\newtheorem{thm}{Theorem}[section]
\newtheorem{lem}[thm]{Lemma}        

\newtheorem{conj}[thm]{Conjecture}
\newtheorem{exam}{Example}

\newtheorem*{abspr*}{Abstract Proposition}
\newtheorem*{abslemma*}{Abstract Lemma}
\newtheorem*{mthm*}{Main Theorem}
\newtheorem*{thm*}{Theorem}
\newtheorem*{defi*}{Definition}
\newtheorem*{lem*}{Lemma}        
\newtheorem*{cor*}{Corollary}
\newtheorem*{prop*}{Proposition}
\newtheorem{rem}{{\it Remark}} 
\newtheorem*{conj*}{Conjecture}

\theoremstyle{plain} 
\newcommand{\thistheoremname}{}
\newtheorem*{genericthm*}{\thistheoremname}
\newenvironment{namedthm*}[1]
{\renewcommand{\thistheoremname}{#1}%
	\begin{genericthm*}}
	{\end{genericthm*}}

\usepackage[leqno]{amsmath}
\makeatletter
\newcommand{\leqnomode}{\tagsleft@true\let\veqno\@@leqno}
\newcommand{\reqnomode}{\tagsleft@false\let\veqno\@@eqno}
\makeatother

\begin{document}
	\title{Muller type irreducibility criterion for generalized principal series }
	\author{CAIHUA LUO}
	\address{Department of Mathematical Sciences, Chalmers University of Technology and the University of Gothenburg, Chalmers Tv\"{a}rgata 3, SE-412 96 G\"{o}teborg}
	\email{caihua@chalmers.se}
	\date{}
	\begin{abstract}
		We obtain an irreducibility criterion for generalized principal series, extending known and frequently employed results for principal series. Our approach rests on a newly observed semi-direct product decomposition of the relative Weyl group and its action on generalized principal series, in conjunction with the theory of Jacquet module. We thus circumvent the obstacle that the Weyl group of a general Levi subgroup is not a Coxeter group. Our statements on irreducibility are formulated in terms of a subgroup of the Knapp--Stein R-group, which arises naturally from our decomposition of the Weyl group. The novel subgroup allows us to state a conjecture on general parabolic induction of arbitrary irreducible admissible representations as opposed to generalized principal series, which are associated with supercuspidal ones.
	\end{abstract}
	\maketitle

\section*{introduction}
One of the aspects of the Langlands program concerns harmonic analysis, especially the analysis of the constituents of parabolic inductions following Harish-Chandra's ``philosophy of cusp forms''. Lots of mathematicians have devoted their efforts to this question, like Langlands, Jacquet, Casselman, Knapp--Stein, Bernstein--Zelevinsky, Shahidi, Speh--Vogan, Silberger, Moeglin--Tadi\'{c} etc.

To be precise, let $G$ be a reductive group over a non-archimedean local field $F$ of characteristic 0. Following from Harish-Chandra's ``philosophy of cusp forms'', there are two main problems in the study of irreducible admissible representations of $G$:
\begin{enumerate}[(i)]
	\item The description of supercuspidal representations of $G$.
	\item The study of generalized principal series. More precisely, let $P=MN$ be a parabolic subgroup of $G$ with $M$ its Levi subgroup, and $\sigma$ be a supercuspidal representation of $M$. Denote by $Ind^G_P(\sigma)$ the normalized parabolic induction from $P$ to $G$ (the so-called generalized principal series). The problem is to study the conditions of irreducibility of $Ind^G_P(\sigma)$ and the constituent structure if reducible. 
\end{enumerate} 
In this paper, we provide an unsurprising, but unexploited irreducibility criterion for generalized principal series. Instead of using the classical Knapp--Stein R-group, our main result is formulated by adopting a modified $R$-group which arises naturally as a subgroup of the Knapp--Stein R-group (please see Lemma \ref{rgp} for the detail).
\begin{namedthm*}{Main Theorem}
	The following two statements are equivalent:
	\begin{enumerate}[(i)]
		\item $Ind^G_P(\sigma)$ is irreducible.
		\item $R_{\sigma}=\{1\}$ and all co-rank one inductions $Ind_{P\cap M_\alpha}^{M_\alpha}(\sigma)$ are irreducible for $\alpha\in \Phi_M^0$. 
	\end{enumerate}
\end{namedthm*}
In the case of $G=GL_n$, this is Bernstein's irreducibility criterion \cite[Theorem 4.2]{bernstein1977induced}. When $\sigma$ is unitary supercuspidal, this is the famous Knapp--Stein R-group theory (cf. \cite{knapp1975singular,silberger1978knapp}). When $\sigma$ is regular supercuspidal, it is well-known in \cite{rodier1981decomposition} (for principal series) and \cite[Theorem 5.4.3.7]{silberger2015introduction} (for generalized principal series). When $\sigma$ is unitary regular supercuspidal, the induction is always irreducible by a theorem of Bruhat in \cite[Theorem 6.6.1]{casselman1995introduction}. Further conditional results for principal series were given in \cite{casselman1980unramified,rodier1981decomposition,winarsky1978reducibility,keys1982decomposition,keys1982reducibility,tadic1994representations}, but a complete answer for principal series was given first by Muller \cite{muller1979integrales} and later by Kato via a different method \cite{kato1982irreducible}. For this reason, we refer to our Main Theorem as a Muller type irreducibility criterion. We would like to mention that Speh--Vogan had a clean, beautiful irreducibility criterion for (regular) generalized principal series of reductive Lie groups about 40 years ago (cf. \cite{speh1977reducibility,speh1980reducibility}).

As an application, our Main theorem provides us a way to prove Lapid--Tadi\'{c}'s conjecture \cite[Conjecture 1.3]{lapid2017some} which we hope to address in a future work. At last, we want to mention that the novelties of the proof of the Main Theorem are the following two key observations:
\begin{enumerate}[(i)]
	\item (cf. Lemma \ref{key1}) \[W_M=W_M^0\rtimes W_M^1. \]
	\item (cf. Lemma \ref{key2}) For $w\in W_M^0$ and $w_1\in W_M^1$, we have, 
	\[Ind^G_P(\sigma)^w\simeq Ind^G_P(\sigma)^{ww_1}. \]
\end{enumerate}
Very recently, we learned that the first observation is an old result of Lusztig and has since been applied creatively in \cite[Lemma 5.2]{lusztig1976coxeter}, \cite[Corollary 2.3]{howlett1980induced}, \cite{morris1993tamely} and \cite{brink1999normalizers}. Furthermore, in view of the semi-direct product decomposition of the Knapp--Stein $R$-group (see Lemma \ref{rgp}), we want to propose the following ambitious conjecture. Let $\tau$ be an irreducible admissible representation of $M$ of the parabolic subgroup $P=MN$ of $G$, we define the modified group $R_\tau$ of $\tau$ in $G$ to be our modified $R$-group of its supercuspidal support.
\begin{namedthm*}{Main Conjecture} (i.e. Conjecture \ref{gmaxim})
	Assume $R_\tau=\{1\}$. Then the following are equivalent: 
	\begin{enumerate}[(i)]
		\item $Ind^G_P(\tau)$ is irreducible.
		\item Co-rank one inductions are irreducible, i.e. $Ind^{M_\alpha}_{M_\alpha\cap P}(\tau)$ is irreducible for any $\alpha\in \Phi_M$.
	\end{enumerate} 
\end{namedthm*}
Very recently, we learned that Jantzen had proved our Main Conjecture for essentially discrete series representation $\tau$ of a Levi group of split groups $Sp_{2n}$ and $SO_{2n+1}$ about 20 years ago for which $R_\tau$ is always trivial (see \cite[Theorem 3.3]{jantzen1996reducibility}). Inspired by his beautiful argument, we would like to see how far we could move forward in our future work.

Let us end the introduction by saying briefly the structure of the paper. In Section 1, we recall some basic notions. In Section 2, we formulate our irreducibility criterion for generalized principal series in terms of Muller's irreducibility criterion for principal series, while in the last, we prepare some necessary observations/facts which play an essential role in the proof and emphasize a history on some special cases of the irreducibility criterion. In Section 3, we give a proof of our Main theorem via Jacquet module argument and the theory of intertwining operator which provides us a third proof of Muller's irreducibility criterion for principal series \cite{muller1979integrales,kato1982irreducible}. In the last section, inspired by our irreducibility criterion, we propose an ambitious irreducibility criterion conjecture and discuss some supportive examples.

		\paragraph*{\textbf{Acknowledgements}}
		I am much indebted to Professor Wee Teck Gan for his guidance and numerous discussions on various topics. I would like to thank Martin Raum for his kind help and support, and thank Maxim Gurevich for his seminar talk on a conjectural criterion of the irreducibility of parabolic inductions for $GL_n$ in the National University of Singapore which rekindles our enthusiasm to explore the mysterious internal structure of parabolic induction.
\section{preliminaries}
Let $G$ be a connected reductive group defined over a non-archimedean local field $F$ of characteristic 0. Denote by $|-|_F$ the absolute value, by $\mathfrak{w}$ the uniformizer and by $q$ the cardinality of the residue field of $F$. Fix a minimal parabolic subgroup $B=TU$ of $G$ with $T$ a minimal Levi subgroup and $U$ a maximal unipotent subgroup of $G$, and let $P=MN$ be a standard parabolic subgroup of $G$ with $M$ the Levi subgroup and $N$ the unipotent radical.

Let $X(M)_F$ be the group of $F$-rational characters of $M$, and set 
\[\mathfrak{a}_M=Hom(X(M)_F,\mathbb{R}),\qquad\mathfrak{a}^\star_{M,\mathbb{C}}=\mathfrak{a}^\star_M\otimes_\mathbb{R} \mathbb{C}, \]
where
\[\mathfrak{a}^\star_M=X(M)_F\otimes_\mathbb{Z}\mathbb{R} \]
denotes the dual of $\mathfrak{a}_M$. Recall that the Harish-Chandra homomorphism $H_P:M\longrightarrow\mathfrak{a}_M$ is defined by
\[q^{\left< \chi,H_P(m)\right>}=|\chi(m)|_F \] 
for all $\chi\in X(M)_F$.

Next, let $\Phi$ be the root system of $G$ with respect to $T$, and $\Delta$ be the set of simple roots determined by $U$. For $\alpha\in \Phi$, we denote by $\alpha^\vee$ the associated coroot, and by $w_\alpha$ the associated reflection in the Weyl group $W=W^G$ of $T$ in $G$ with
\[W=N_G(T)/T=\left<w_\alpha:~\alpha\in\Phi\right>. \]
Denote by $w_0^G$ the longest Weyl element in $W$, and similarly by $w_0^M$ the longest Weyl element in the Weyl group $W^M$ of the Levi subgroup $M$. 

Likewise, we denote by $\Phi_M$ the set of reduced relative roots of $M$ in $G$, by $\Delta_M$ the set of relative simple roots determined by $N$ and by $W_M:=N_G(M)/M$ the relative Weyl group of $M$ in $G$. In general, a relative reflection $\omega_\alpha:=w_0^{M_\alpha}w_0^M$ with respect to a relative root $\alpha$ does not preserve our Levi subgroup $M$. Denote by $\Phi^0_M$ (resp. $(\Phi_M^{0})^+$) the set of those relative (resp. positive) roots which contribute reflections in $W_M$. It is easy to see that $W_M$ preserves $\Phi_M$, and further $\Phi_M^0$ as well, as $\omega_{w.\alpha}=w\omega_\alpha w^{-1}$. Note that $W_M$ in general is larger than the one generated by those relative reflections.

For our purpose, we define the ``small'' relative Weyl group $W_M^0\subset W_M$ to be the one generated by those relative reflections, i.e.
\[W_M^0:=\left<w_\alpha:~\alpha\in \Phi^0_M \right>. \]
Denote by $\Delta_M^0$ the relative simple roots of $\Phi_M^0$. 

Recall that the canonical pairing $$\left<-,-\right>:~\mathfrak{a}^\star_M\times \mathfrak{a}_M\longrightarrow\mathbb{Z}$$ suggests that each  $\alpha\in \Phi_M$ will enjoy a one parameter subgroup $H_{\alpha^\vee}(F^\times)$ of $M$ satisfying: for $x\in F^\times$ and $\beta\in \mathfrak{a}^\star_M$,
\[\beta_\alpha(x):=\beta(H_{\alpha^\vee}(x))=x^{\left<\beta,\alpha^\vee\right>}. \]

For $P=MN$ a parabolic subgroup of $G$ and an admissible representation $(\sigma,V_\sigma)~ (resp.~(\pi,V_\pi))$ of $M~(resp.~G)$, we have the following normalized parabolic induction from $P$ to $G$ which is a representation of $G$
\[Ind_P^G(\sigma):=\{\mbox{smooth }f:G\rightarrow V_\sigma|~f(nmg)=\delta_P(m)^{1/2}\sigma(m)f(g), \forall n\in N, m\in M~and~g\in G\}, \]
where $\delta_P$ stands for the modulus character of $P$, i.e., denote by $\mathfrak{n}$ the Lie algebra of $N$,
\[\delta_P(nm)=|det~Ad_\mathfrak{n}(m)|_F, \]
and the normalized Jacquet module $J_M(\pi)$ with respect to $P$ as a representation of $M$ is defined by
\[\pi_N:=V/\left<\pi(n)e-e:~n\in N,e\in V_\pi\right>. \] 
Given an irreducible unitary admissible representation $\sigma$ of $M$ and $\nu\in \mathfrak{a}^\star_{M}$, let $I_P^G(\nu,\sigma)$ be the representation of $G$ induced by $\sigma$ and $\nu$ as follows:
\[I^G_P(\nu,\sigma)=Ind^G_P(\sigma\otimes \nu)=Ind_P^G(\sigma\otimes q^{\left<\nu,H_P(-)\right>}) .\]
Define the action of $w\in W_M$ on representations $\sigma\otimes \nu$ of $M$ to be $w.(\sigma\otimes \nu):=(\sigma\otimes \nu)\circ Ad(w^{-1})$ and $(\sigma\otimes \nu)^w:=(\sigma\otimes \nu)\circ Ad(w)$.
\section{generalized principal series}
In this section, we first revisit Muller's irreducibility criterion for principal series of split groups and reformulate it for generalized principal series of general $G$, then recall some history concerning some special cases of such criterion and prepare some necessary structure theory for later use.

In \cite{muller1979integrales}, she defines a subgroup $W_\lambda^1$ of the Weyl group $W$ governing the reducibility of the ``unitary'' part of principal series on the Levi level, which is indeed the Knapp--Stein $R$-group as follows (cf. \cite{winarsky1978reducibility,keys1982decomposition}), for the principal series $Ind^G_B(\lambda)$ of a split group $G$,
\begin{align*}
\Phi_{\lambda}^0&:=\{\alpha\in \Phi:~\lambda_\alpha=Id \},\\
W_{\lambda}^0&:=\left<w_\alpha:~\alpha\in \Phi_\lambda^0 \right>,\\
W^1_\lambda&:=\{w\in W_\lambda:~w.(\Phi_\lambda^0)^+>0 \},\\
W_\lambda&:=\{w\in W:~w.\lambda=\lambda \}.
\end{align*} 
In view of \cite[Lemma I.1.8]{waldspurger2003formule}, one has
\[W_\lambda=W_\lambda^0\rtimes W_\lambda^1. \]
Following the Knapp--Stein R-group theory (cf. \cite{silberger2015introduction}), we insist to denote by $R_\lambda$ the subgroup $W_\lambda^1$.

In order to generalize the above notions for generalized principal series, we modify some of the notions in what follows. Recall that given a parabolic subgroup $P=MN$ of $G$, a unitary supercuspidal representation $\sigma$ of $M$ and an unramified character $\nu$ of $M$ in $\mathfrak{a}_M^\star$, one forms a parabolic induction
\[I^G_P(\nu,\sigma):=Ind_P^G(\sigma\otimes \nu). \]
Recall that for $\alpha\in \Phi_M^0$, the associated refection $w_\alpha$ is defined as $w^{M_\alpha}w^M$, where $M_\alpha$ is the co-rank one Levi subgroup determined by $\alpha$, and $w^M (resp.~w^{M_\alpha})$ is the longest Weyl element in the Weyl group $W^M$ (resp. $W^{M_\alpha}$) of $M$ (resp. $M_\alpha$). Also recall that the relative Weyl group $W_M$ of M in G is defined to be
\[W_M:=N_G(M)/M=\{w\in W:~w.M=M \}/W^M (\mbox{cf. \cite[P. 19 \& 104]{silberger2015introduction} or \cite{moeglin_waldspurger_1995}}), \]
and the ``small'' relative Weyl group $W_M^0$ is 
\[W_M^0:=\left<w_\alpha:~\alpha\in \Phi_M^0 \right>, \]
where $\Phi_M^0$ is the set of those reduced relative roots $\alpha$ which contribute a reflection $w_\alpha$ preserving $M$, i.e. $w_\alpha.M=M$. Given these, we can define the analogous notions as follows:
\begin{align*}
\Phi_{\sigma_\nu}^0&:=\{\alpha\in \Phi_M^0:~w_\alpha.(\sigma\otimes \nu)=(\sigma\otimes\nu) \},\\
W_{\sigma_\nu}^0&:=\left<w_\alpha:~\alpha\in \Phi_{\sigma_\nu}^0 \right>,\\
W^1_{\sigma_\nu}&:=\{w\in W_{\sigma_\nu}:~w.(\Phi_{\sigma_\nu}^0)^+>0 \},\\
W_{\sigma_\nu}&:=\{w\in W_M:~w.(\sigma\otimes \nu)=(\sigma\otimes \nu) \}.
\end{align*}
Likewise, via \cite[Lemma I.1.8]{waldspurger2003formule}, we have
\[W_{\sigma_\nu}=W_{\sigma_\nu}^0\rtimes W_{\sigma_\nu}^1, \]
and we denote $R_{\sigma_\nu}$ to be $W_{\sigma_\nu}^1$ following tradition, but it is not the exact $R$-group in the sense of Silberger for generalized principal series, even for principal series, for example 
\[Ind^{SL_2}_T(\chi) ~with ~\chi^2=1 ~but~ \chi\neq 1,\]
in such case, we know that 
\[R_\lambda\simeq \mathbb{Z}/2\mathbb{Z}, ~but~ R_{\sigma_\nu}=\{1\}.\]

As all are well-prepared, now we can state the Muller type irreducibility criterion for generalized principal series as follows:
\begin{mthm*}
	Keep the notions as before. The following two statements are equivalent
	\begin{enumerate}[(i)]
		\item $I^G_P(\nu,\sigma)$ is irreducible.
		\item $R_{\sigma_\nu}=\{1\}$ and all co-rank one inductions $Ind_{P\cap M_\alpha}^{M_\alpha}(\sigma\otimes \nu)$ are irreducible for $\alpha\in \Phi_M^0$. 
	\end{enumerate}
\end{mthm*}
Before turning to the proof in the next section, we first observe some facts which play an essential role later on in what follows.

\begin{lem}\label{key1}
	Keep the notation as above. We have
	\[W_M=W_M^0\rtimes W_M^1, \]
	where $W_M^1$ is defined to be
	\[W_M^1:=\{w\in W_M:~w.(\Phi_M^0)^+>0 \}. \]
	Moreover $\Phi_M^0$ is a relative subroot system, may not be irreducible.
\end{lem}
\begin{proof}
	The facts that
	\[W_M^0\lhd W_M, \]
	and
	\[\mbox{$\Phi_M^0$ is a relative subroot system} \]
	are easy corollaries of the following observation
	\[ww_\alpha w^{-1}=w_{w.\alpha}\mbox{ for }\alpha\in \Phi_M^0\mbox{ and }w\in W_M. \]
	The remaining part follows from the same argument for the definition of the Knapp--Stein $R$-group as above (or cf. \cite[Lemma A.2]{luo2018R}).
\end{proof}
As the nature of the Jacquet module argument is to vastly use the induction by stage property of parabolic induction, so we need the following lemma.

Recall that $Z_M$ is the center of $M$, where $M$ is the Levi subgroup of the parabolic subgroup $P=MN$ in $G$. For $\alpha\in \Phi_M^0$, one has the associated coroot $\alpha^\vee$, then we define 
\[\nu_\alpha(x)=\nu(\alpha^\vee(x))\mbox{ for $x\in F^\times$},\] and define
\[\Delta_1:=\{\alpha\in \Phi_M^0:~\nu_\alpha=1 \}. \]
It is easy to see that $\Delta_1$ also forms a relative subroot system. Denote by $W_{\Delta_1}$ the Weyl group generated by $\Delta_1$. Then we have, similar to \cite[Lemma 4.1]{muller1979integrales},
\begin{lem}\label{urd}
	\leavevmode
	\begin{enumerate}[(i)]
		\item $\Delta_1$ admits a base of relative simple roots which is part of the counterpart for $\Phi_M^0$.
		\item $W_{\sigma_\nu}$ is a subgroup of $W_{\Delta_1}$.  
	\end{enumerate}
\end{lem}
\begin{proof}
	They follow from the fact that 
	\[\mbox{$\Phi_M^0$ is a relative subroot system.} \]
\end{proof}

Observe that $Ind_M^{M_{\Delta_1}}(\sigma\otimes\nu)\simeq Ind_M^{M_{\Delta_1}}(\sigma)\otimes\nu$ is unitary after twisting by the central character $\nu$ of $M_{\Delta_1}$. In view of Lemma \ref{urd}, it is quite natural to define the Knapp--Stein $R$-group of $(G,~I^G_P(\nu,\sigma))$ in terms of the $R$-group of $(M_{\Delta_1},~Ind^{M_{\Delta_1}}_M(\sigma\otimes \nu))$, i.e. $(M_{\Delta_1},~Ind^{M_{\Delta_1}}_M(\sigma))$. Then the next question is to see how far $R_{\sigma_\nu}$ differs from the Knapp--Stein $R$-group. Like the Knapp--Stein R-group theory, we define $W_{\sigma_\nu}^{0'}$ to be the normal subgroup of $W_{\sigma_\nu}^0$ which governs the unitary co-rank one irreducibility of $I^G_P(\nu,\sigma)$, i.e.
\[W_{\sigma_\nu}^{0'}:=\left<w_\alpha\in \Phi_{\sigma_\nu}^0: \mu_\alpha(\sigma)=0 \right>, \]
where $\mu_\alpha(-)$ is the co-rank one Plancherel measure associated to $\alpha$ (please refer to \cite[Section V.2]{waldspurger2003formule} for details).
Thus we have
\[W_{\sigma_\nu}=W_{\sigma_\nu}^{0'}\rtimes R'_{\sigma_\nu}, \]
which in turn implies that  
\begin{lem}\label{rgp}
	Keep the notions as above. We have
	\[R'_{\sigma_\nu}=R_{\sigma_\nu}^0\rtimes R_{\sigma_\nu}, \]
	where 
	\[R_{\sigma_\nu}^0\simeq W_{\sigma_\nu}^0/W_{\sigma_\nu}^{0'}\hookrightarrow R'_{\sigma_\nu}.\]
\end{lem}
\begin{rem}
	It is easy to see that $R'_{\sigma_\nu}$ is exactly the Knapp--Stein $R$-group when the parabolic induction datum $\sigma\otimes \nu$ is unitary.
\end{rem}

\begin{lem} \label{key2}
	For $w\in W_M^0$ and $w_1\in W_M^1$, we have
	\[I^G_P(\nu,\sigma)^w\simeq I^G_P(\nu,\sigma)^{ww_1}. \]
\end{lem}
\begin{proof}
	it reduces to show 
	\[I^G_P(\nu,\sigma)\simeq I^G_P(\nu,\sigma)^{w_1}, \]
	which follows from the associativity property of intertwining operators (cf. \cite[IV.3.(4)]{waldspurger2003formule}). To be precise, up to non-zero scalar, the non-trivial intertwining operator
	\[A: ~Ind_P^G(\sigma\otimes \nu)\longrightarrow Ind_P^G(\sigma\otimes \nu)^{w_1} \]
	is equal to 
	\[J_{P|P^{w_1}}(\sigma^{w_1}\otimes\nu^{w_1})\circ\lambda(w_1):~Ind_P^G(\sigma\otimes \nu)\longrightarrow Ind^G_{P^{w_1}}(\sigma\otimes\nu)^{w_1}\longrightarrow Ind_P^G(\sigma\otimes\nu )^{w_1}.  \]
	By \cite[IV.3.(4)]{waldspurger2003formule}, we have
	\[J_{P|P^{w_1}}(\sigma^{w_1}\otimes\nu^{w_1})J_{P^{w_1}|P}(\sigma\otimes\nu)=\prod j_\alpha(\sigma\otimes\nu) J_{P|P}(\sigma\otimes\nu), \]
	where $\alpha$ runs over $\Phi_M(P)\cap \Phi_M(\overline{P^{w_1}})$ with $\overline{P^{w_1}}$ the opposite parabolic subgroup of $P^{w_1}$. Notice that 
	\[w_1.(\Phi_M^0)^+>0, \]
	so we have
	\[\Phi_M(P)\cap \Phi_M(\overline{P^{w_1}})\cap \Phi_M^0=\emptyset .\]
	In view of \cite[Corollary 1.8]{silberger1980special}, for $\alpha\in \Phi_M(P)-\Phi_M^0$, $j_\alpha(\sigma\otimes \nu)\neq 0,\infty$, and the associated co-rank one induced representation is always irreducible. Thus
	\[\prod j_\alpha(\sigma\otimes\nu) J_{P|P}(\sigma\otimes\nu)=\prod j_\alpha(\sigma\otimes\nu)\neq 0,~\infty. \]
	Whence $A$ is an isomorphism. 
	
\end{proof}
Let us end this section by recalling the history on some special cases of the Main Theorem, especially the regular case and the unitary case. Some irreducibility conditions under various different assumptions for principal series were given in \cite{casselman1980unramified,rodier1981decomposition,winarsky1978reducibility,keys1982decomposition,keys1982reducibility,muller1979integrales,kato1982irreducible,tadic1994representations} 
\begin{thm}(cf. \cite[Theorem 5.4.3.7]{silberger2015introduction})\label{sils}
	If the inducing datum $\sigma\otimes \nu$ is {\bf regular}, i.e. $W_{\sigma_\nu}=\{1\}$. Then 
	\[I^G_P(\nu,\sigma)\mbox{ is irreducible iff }Ind_{P\cap M_\alpha}^{M_\alpha}(\sigma\otimes \nu)\mbox{ is irreducible for all }\alpha\in \Phi_M^0. \]
\end{thm}
\begin{thm}(cf. \cite{silberger1978knapp})\label{silu}
	If the inducing datum $\sigma\otimes \nu$ is {\bf unitary}, i.e. $\nu=0$. Then
	\[I^G_P(\nu,\sigma)\mbox{ is irreducible iff }R'_{\sigma_\nu
	}=\{1\}~\mbox{iff}~R_{\sigma_\nu}=\{1\}~ + \mbox{ co-rank one irreducibility}. \]	
\end{thm}
At last, for the unitary regular case, it is always irreducible which is a theorem of Bruhat in \cite[Theorem 6.6.1]{casselman1995introduction}.

\section{proof of the irreducibility criterion}
In this section, we carry out the proof of the Main theorem, i.e. the irreducibility criterion for generalized principal series, following Casselman--Tadi\'{c}'s Jacquet module argument. Let us first recall the Main theorem as follows:
\begin{mthm*}[Muller type irreducibility criterion]
	$I^G_P(\nu,\sigma)$ is irreducible if and only if the following are satisfied
	\begin{enumerate}[(i)]
		\item $R_{\sigma_\nu}=\{1\}$
		\item $Ind_{P\cap M_\alpha}^{M_\alpha}(\sigma\otimes \nu)$ is irreducible for any $\alpha\in \Phi_M^0$.
	\end{enumerate}
\end{mthm*}
\begin{proof}
	For the necessary part, we first show that the co-rank one inductions are irreducible. To be precise, for each $\alpha\in \Phi_M^0$, under the conjugation of a relative Weyl element $w\in W_M$, we may assume that $\alpha$ is a relative simple root. Therefore 
	\[Ind_{P\cap M_\alpha^w}^{M^w_\alpha}(\sigma\otimes \nu)^w\mbox{ is irreducible,} \]
	which follows from the fact that $I_P^G(\nu,\sigma)$ and $I_{P}^G(\nu,\sigma)^w$ share the same constituent. Which in turn implies that $Ind_{P\cap M_\alpha}^{M_\alpha}(\sigma\otimes\nu)$ is irreducible.
	
	Now it remains to show that our modified R-group $R_{\sigma_\nu}$ is trivial. Note that $W_{\sigma_\nu}$ is a subgroup of the Weyl group $W_{\Delta_1}$ of the Levi subgroup $M_{\Delta_1}$ determined by $\Delta_1$ (see Lemma \ref{urd}), a same argument as above shows that 
		\[R_{\sigma_\nu}=\{1\}, \]
	given the fact that we can move out the $\nu$ from the inducing data on the $M_{\Delta_1}$-level as follows:
	\[Ind_{P\cap M_{\Delta_1}}^{M_{\Delta_1}}(\sigma\otimes \nu)\simeq Ind_{P\cap M_{\Delta_1}}^{M_{\Delta_1}}(\sigma)\otimes \nu .\]
	
	As for the sufficient part, one can follow Muller's intertwining operator argument for principal series based on the following observation which is a corollary of Casselman's subrepresentation theorem, which has not been pointed out clearly in \cite{muller1979integrales},
	\[I^G_P(\nu,\sigma)\mbox{ is irreducible iff }I^G_P(\nu,\sigma)\simeq I^G_P(\nu,\sigma)^w\mbox{ and }Hom_G(I^G_P(\nu,\sigma), I^G_P(\nu,\sigma)^w)\simeq \mathbb{C}	\mbox{ for all }w\in W_M. \]
	But we will give an intuitive argument using the Jacquet module machine and the theory of intertwining operator as follows. Without loss of generality, for $\pi\in JH(I^G_P(\nu,\sigma))$, assume $\sigma\otimes \nu\in J_M(\pi)$, then
	\[I^G_P(\nu,\sigma)\mbox{ is irreducible iff }(\sigma\otimes\nu)^w\in J_M(\pi)\mbox{ for all }w\in W_M. \]
	To show that $(\sigma\otimes\nu)^w\in J_M(\pi)\mbox{ for all }w\in W_M$, we divide it into the following two steps:
		\begin{enumerate}[(i)]
			\item We first show that the multiplicity appears in $J_M(\pi)$, i.e.
			\[|W_{\sigma_\nu}|(\sigma\otimes \nu)\subset J_M(\pi).\]
			By Lemma \ref{urd}, we have the essentially unitary induction on $M_{\Delta_1}$-level, i.e.
			\[Ind_{P\cap M_{\Delta_1}}^{M_{\Delta_1}}(\sigma\otimes \nu)\simeq Ind_{P\cap M_{\Delta_1}}^{M_{\Delta_1}}(\sigma)\otimes \nu .\]
			As $R_{\sigma_\nu}=1$ and all co-rank one inductions are irreducible, thus Theorem \ref{silu} says that the above induction is irreducible, whence our claim holds. 
			\item We then show that each orbit appears in $J_M(\pi)$, i.e.
			\[(\sigma\otimes \nu)^w\in J_M(\pi)\mbox{ for all }w\in W_M/W_{\sigma_\nu}, \]
			which follows from the case of regular inducing data after conjugating by a Weyl element in $W_M$ which is hidden in \cite{luo2018R}). To be precise, as $$W_M=W_M^0\rtimes W_M^1,$$ a similar argument as in the proof of the necessary part shows that
			\[(\sigma\otimes \nu)^w\in J_M(\pi) \]
			for all $w\in W_M^0/W_{\sigma_\nu}^0$.
			
			Therefore it remains to show that, for any $w\in W_M^1$ which is not in $W_{\sigma_\nu}$,
			\[(\sigma\otimes \nu)^w\in J_M(\pi). \]
			This follows from our key Lemma \ref{key2}.
		\end{enumerate}
\end{proof}
\begin{rem}
In view of the above argument, it is easy to see that the irreducibility criterion holds for covering groups given the fact that the Knapp--Stein $R$-group theory has been established in \cite{luo2017R} and the following irreducibility criterion for regular generalized principal series of covering groups.
\end{rem}
Recall that $\tilde{G}$ is a finite central covering group of $G$, and $\tilde{P}=\tilde{M}N$ is a parabolic subgroup of $\tilde{G}$. Let $\tilde{\sigma}$ be a genuine regular supercuspidal representation of $\tilde{M}$, and denote by $Ind^G_P(\tilde{\sigma})$ the normalized parabolic induction of $\tilde{\sigma}$ from $\tilde{P}$ to $\tilde{G}$. All other notions are the same as in the non-cover case. Then we have
\begin{lem}The following are equivalent
	\begin{enumerate}[(i)]
		\item $Ind^{\tilde{G}}_{\tilde{P}}(\tilde{\sigma})$ is irreducible.
		\item $Ind_{\tilde{P}\cap \tilde{M}_\alpha}^{\tilde{M}_\alpha}(\tilde{\sigma})$ is irreducible for all $\alpha\in \Phi_M^0$.
	\end{enumerate}
\end{lem}
\begin{proof}
	This follows from the properties of Plancherel measure generalizing to covering groups in \cite{luo2017R} and the argument in the above Lemma \ref{key2}.
\end{proof}	

\section{a conjectural criterion of parabolic inductions}
In the section, we would like to first serve you two simple observations, originating from \cite{rodier1981decomposition,luo2018R,luo2018C}, on a conjectural irreducibility criterion of parabolic induction learned from M. Gurevich's talk in the National University of Singapore. Then we would like to propose an ambitious conjecture for general groups.

In the following, let us first recall the explicit conjectural irreducibility criterion given in M. Gurevich's talk. Let $P=MN$ be a parabolic subgroup of $GL_n$ with the Levi subgroup $M=\prod_{i\in I}GL_{n_i}$, and denote by $P_{i,j}$ the parabolic subgroup of $GL_{(n_i+n_j)}$ with the Levi subgroup $M_{i,j}=GL_{n_i}\times GL_{n_j}$ for $i\neq j\in I$. For an irreducible admissible representation $\otimes_i\sigma_i$ of $M$, the conjectural irreducibility criterion of parabolic induction for $GL_n$ is as follows:
\begin{conj}\label{maxim}
	The parabolic induction $\times_i\sigma_i:=Ind^{GL_n}_P(\otimes_i\sigma_i)$ is irreducible if and only if the co-rank one parabolic induction $\sigma_i\times \sigma_j:=Ind^{GL_{(n_i+n_j)}}_{M_{i,j}}(\sigma_i\otimes \sigma_j)$ is irreducible for all $i\neq j\in I$.
\end{conj}
The first observation comes from the structure theory of regular generalized principal series (cf. \cite{luo2018R,rodier1981decomposition}) which roughly says that if the supercuspidal support of our induction data is regular, then the above conjecture holds. To be more precise, write the supercuspidal support of $\sigma_i$ as 
\[\{\tau_{i,k} \}_{k=1}^{l_i}\mbox{ with } \tau_{i,k}\mbox{ supercuspidal}. \] 
Then our simple observation can be stated as follows.
\begin{lem}\label{regular}
	Assume $\{\tau_{i,k} \}_{i,k}$ is regular, i.e. $\otimes_{i,k}\tau_{i,k}$ is regular. Then Conjecture \ref{maxim} holds.
\end{lem}
\begin{proof}
	The necessary part is obvious, one only has to prove the sufficient part. If $\sigma_i\times\sigma_j$ is irreducible, then 
	\[\tau_{i,k_1}\times \tau_{j,k_2} \mbox{ is irreducible for all }k_1,~ k_2, \]
	which follows from \cite[Theorem 3.7]{luo2018R}). Applying \cite[Theorem 3.7]{luo2018R} again, one knows that 
	\[\times_i\sigma_i \mbox{ is irreducible},\]
	whence Conjecture \ref{maxim} holds.
\end{proof}
\begin{rem}
	If one replaces the condition ``$\sigma_i\times \sigma_j$ is irreducible'' by ``co-rank one induction is irreducible'' in Conjecture \ref{maxim}, it is easy to see that the Lemma \ref{regular} still holds for general connected reductive groups as \cite[Theorem 3.7]{luo2018R} applies to such generality.
\end{rem}
The second observation comes from a ``product formula'' in \cite{luo2018C} which roughly says that if the co-rank one reducibility and the Knapp--Stein $R$-group conditions lie in a Levi subgroup $L$ of $G$, then 
\[\#JH(Ind^L_{L\cap P}(\sigma))=\#JH(Ind^G_P(\sigma)), \]
where $\sigma$ is a supercuspidal representation of the Levi subgroup $M$ of $P$. 

To be precise, denote by $\Theta_Q$ the associated subset of $\Delta$ which determines the Levi subgroup $Q$ of $G$. We decompose $\Theta_L=\Theta_1\sqcup\cdots \sqcup \Theta_t $ into irreducible pieces, and accordingly $\Theta_M=\Theta_1^M\sqcup \cdots \sqcup \Theta_t^M$. Assume that $R_\sigma$ decomposes into $R_\sigma=R_1\times \cdots \times R_t$ with respect to the decomposition of $\Theta_L$, and a similar decomposition pattern holds for the co-rank one reducibility, i.e. co-rank one reducibility only occurs within $P_{\Theta_i}=M_{\Theta_i}N_{\Theta_i}$ for $1\geq i\leq t$. Then we have  
\begin{namedthm*}{Product formula}(cf. \cite[Corollary 2.2]{luo2018C})
	\[\#(JH(Ind^G_P(\sigma)))=\prod_{i=1}^{t}\#(JH(Ind_{M_{\Theta_i^M}}^{M_{\Theta_i}}(\sigma))).\]
\end{namedthm*}
So an easy corollary of the Product formula is
\begin{lem}\label{prodlem}
	Assume the decomposition pattern of the co-rank one reducibility and our revised $R$-group of the supercuspidal support data $\{\tau_{i} \}_{i}$ of $\sigma$ as above is exactly $L=M$, then Conjecture \ref{maxim} holds.
\end{lem}
\begin{rem}
	We learned that Conjecture \ref{maxim} is now a theorem of Gurevich.
\end{rem}
Let $\sigma$ be an irreducible admissible representation of $M$ of the parabolic subgroup $P=MN$ of $G$, we define the revised group $R_\sigma$ of $\sigma$ in $G$ to be our revised $R$-group of its supercuspidal support. Inspired by Conjecture \ref{maxim} and the above two observations, we would like to propose an ambitious conjecture for connected reductive groups in the following.
\begin{conj}\label{gmaxim}
	Keep the notions as above. Assume $R_\sigma=\{1\}$. Then the following are equivalent 
	\begin{enumerate}[(i)]
		\item $Ind^G_P(\sigma)$ is irreducible.
		\item the co-rank one induction is irreducible, i.e. $Ind^{M_\alpha}_{M_\alpha\cap P}(\sigma)$ is irreducible for any $\alpha\in \Phi_M$.
	\end{enumerate} 
\end{conj}
In what follows, let us discuss some supportive examples. 
\begin{exam}
	If $\sigma$ is supercuspidal, then Conjecture \ref{gmaxim} follows from the Main Theorem, i.e. Muller type irreducibility criterion.
\end{exam}	
For classical groups, we know that $R_\sigma$ is always trivial. One may check also that $R_\sigma$ is trivial for Chevalley groups of types $B_n$ and $C_n$ (cf. \cite{keys1982decomposition}).
\begin{exam}
Let $G$ be a classical group, under the assumption of Lemma \ref{prodlem} with $L=GL$-part of $M$, then Conjecture \ref{maxim} implies Conjecture \ref{gmaxim}. 
\end{exam}
\begin{exam}
	If $G=Sp_{2n}$ or $SO_{2n+1}$, and $\sigma$ is an essentially discrete series representation of $M$, this is a well-known result of Jantzen (see \cite[Theorem 3.3]{jantzen1996reducibility}). In view of Jantzen's argument for $Sp_{2n}$ and $SO_{2n+1}$, one could expect to extend his result to other classical groups and Chevalley groups of type $B_n$ and $C_n$, thus confirm our Conjecture \ref{gmaxim} for those cases.
\end{exam}
Indeed, inspired by Jantzen's beautiful argument in \cite{jantzen1996reducibility}, we would like to check our Conjecture \ref{gmaxim} for $\sigma$ essentially discrete series in our future work.

\bibliographystyle{abbrv}
\bibliography{ref}

			
\end{document}